\documentclass[num-refs]{wiley-article}

\usepackage{amsmath}
 \usepackage{amssymb}
\usepackage{mathrsfs}
\usepackage{mathtools}
\usepackage{subcaption}
\usepackage{wrapfig}
\usepackage{float}
\usepackage{siunitx}
\DeclarePairedDelimiter\fl{\lfloor}{\rfloor} 
\DeclarePairedDelimiter\cb{\{}{\}} 
\papertype{Original Article}

\title{Patterns of primes and composites on divisibility graph}

\author[1]{R. Abiya}
\author[1]{G. Ambika}
\affil[1]{Department of Physics, Indian Institute of Science Education and Research Tirupati, Andhra Pradesh, 517 507, India}

\corraddress{G. Ambika}
\corremail{g.ambika@iisertirupati.ac.in }

\runningauthor{Abiya and Ambika}

\begin{document}

\maketitle

\begin{abstract}
We study the undirected divisibility graph in which the vertex set is a finite subset of consecutive natural numbers up to N. We derive analytical expressions for measures of the graph like degree, clustering, 
geodesic distance and centrality in terms of the floor functions and the divisor functions.  We discuss how these measures depend on the vertex labels and the size of graph N. We also present the specific case of prime vertices separately as corollaries. We could explain the patterns in the local measures for a finite size graph as well as the trends in global measures as the size of the graph increases. 

\end{abstract}

\section{INTRODUCTION AND RESULTS}
In the world of numbers, primes and composites form two non-overlapping infinite sets. Out of these, the prime numbers, which form the building blocks of natural numbers, occur scattered in a non-homogeneous fashion. The famous Prime number theorem which states that the number of primes up to N, approaches $\frac{N}{logN}$, for large N is one of the instances where a pattern in prime numbers is found \cite{Hardy75}. In this study, we take a graph theoretic approach using the framework of a divisibility graph to understand the trends in primes and composites through some of the graph properties. As divisibility pattern among numbers can
distinguish the primes from other numbers in a basic way, the properties of a divisibility graph can probe into the intricacies in the architecture of primes and composites in a natural way. In this work, we present how the properties of the divisibility graph can be related to the properties of natural numbers, especially prime numbers.

We derive expressions for some of the measures like the degree, clustering and connectance or link density of a given vertex, average shortest distance between the vertices and some of the centrality properties of this graph. We find these properties depend mainly on two functions: the divisor function $s(n)$ and the floor function $\fl*{\frac{N}{n}}$.  Hence the expressions derived, specifically for the divisibility graph, help to understand the trends observed in the measures of the graph as its size N increases. We also explain the inherent patterns in degrees and clustering coefficients as arising due to the trends in the divisor functions and floor functions of natural numbers. We could also bring out the specific trends shown by the measures corresponding to prime numbers.

\subsection{Background and notation}
There are many graphs associated with the set of natural numbers. Lewis introduced the prime vertex graph and the common divisor graph \cite{Lewis08}. Praeger and Iranmanesh defined the bipartite divisor graph whose vertex set is a disjoint union of the vertex sets of the prime vertex graph and the common divisor graph \cite{Ir10}. In \cite{Ir16}, the authors study a undirected version of the divisibility graph $\mathscr{D}$(X), as a simple graph with vertex set X* ($=X-\{1\}$) and two elements of X* are adjacent if one of them divides the other. D(X) and $\mathscr{D}$(X) contain multiple components. Also reported is the construction of a weighted bipartite of composite and prime numbers, with connections between them decided by the prime factorisation \cite{Garcia14}. The congruence
relations among numbers are explored using a multiplex graph in reference \cite{XiaoYong16}. The patterns and symmetry present in different graph properties of the divisibility graph for large N, are studied using computational methods \cite{Snehal15}. Following this, a recent work has
applied this to the specific case of divisibility pattern within the elements of Pascal triangles \cite{Solares20}.  

In this work, we study the undirected divisibility graph $G_N$, in which the vertex set is $X_N$ = $\cb*{1,\dots,N}$, the finite subset of consecutive natural numbers up to any chosen number N. Two elements of $X_N$ are adjacent if one of the two divides the other. The Adjacency matrix is $A = [a_{ij}]$ where

$a_{ij}$ = 
$\begin{cases}
1, & \text{if $i \ne j$ and (either $i$ divides $j$ or $j$ divides $i$)} 
\\0, & \text{otherwise} 
\end{cases}$     

$G_N$ contains a single component because of the vertex $1$ which is adjacent to all other elements of $X_N$. For example, consider the graph $G_{20}$ given in figure \ref{fig:1}. In the graph, vertex $1$ is connected to all the other $N-1$ vertices, vertex $2$ is connected to vertex $1$ and all the even numbers, and so on. In all the graphs discussed in this context, the vertex labels are fixed, since the edges between vertices depend on the vertex labels. 

In graph analytics, degree or the number of neighbours, is an important concept in identifying significant vertices in the graph. It is used to measure the importance of a vertex or how central a vertex is in a graph. We discuss the degree ($k_n$) and some other properties like local clustering coefficient ($c_n$), mean geodesic distance ($l_n$) and betweenness centrality ($x_n$) of the vertices $n$ in $G_N$ and the connectance ($C$) of the graph.
\begin{figure}[hbtp]
\centering
\includegraphics[width=0.49\textwidth]{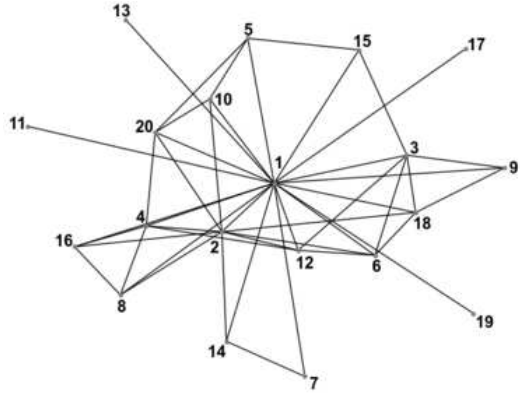}
\caption{ A divisibility graph of natural numbers $G_N$ is defined as a simple, finite graph with $N$ vertices, labelled $1,\dots,N$. There is an edge between two vertices labelled $a$ and $b$ if and only if one of the two divides the other. The above figure shows $G_{20}$. The vertex labels are fixed, since the connections between vertices depend on the labels.}\label{fig:1}
\end{figure}

\subsection{Statement of results}

We discuss the dependence of the properties of $G_N$ on the vertex labels and the number of vertices (N). In particular, the values of these properties when the vertex label is a prime number are discussed as corollaries. This approach also enables us to express the properties of a given divisibility graph without using the Adjacency matrix, but in terms of the divisor function and the number of multiples less than or equal to N.

\subsection{Properties of $G_N$}

\begin{definition}
The degree of a vertex $n$, is the number of edges connected to $n$. It is denoted by $k_n$. 
\end{definition}
\begin{theorem}
Let $n \in X_N$ and $n = \prod_{\alpha = 1}^k p^{j_\alpha}_\alpha$ be the prime factorisation of n. Let $s(n)$ be the number of divisors of $n$ including $n$ itself. $s(n) = \prod_{\alpha = 1}^k (j_\alpha + 1)$. Then,
\begin{equation}
    k_n = \fl*{\frac{N}{n}} + s(n) - 2 
\end{equation}
\end{theorem}
\begin{proof}

$\displaystyle{k_n = \sum_{m=1}^N a_{nm}}$. By definition. $a_{nm}$ \  is non-zero and one whenever m is either a multiple or a divisor of n. Therefore $k_n$ = (No. of multiples of n less than N) + (No. of divisors of n excluding itself). The multiples of n less than N are ($2.n$), ($3.n$), $\dots$, ($\fl*{\frac{N}{n}}.n$) . So, the number of multiples of n less than N is ($ \fl*{\frac{N}{n}} $ - 1). The number of divisors of n excluding itself is $s(n)$ - 1. Hence the proof.
\end{proof}

\begin{corollary}
$k_p = \fl*{\frac{N}{p}}$ if and only if p is a prime. This is due to the fact that s(p) = 2 if and only if p is a prime.
\end{corollary}

We relate the patterns in the degrees $k_n$ of $n$, for different $n$ to the combined patterns or trends in divisor function, $s(n)$ and $\fl*{\frac{N}{n}}$ for different $n$, in figure \ref{fig:2}
   
\begin{figure}[htbp]
\centering
\begin{subfigure}{0.49\textwidth}
\includegraphics[width=1\textwidth]{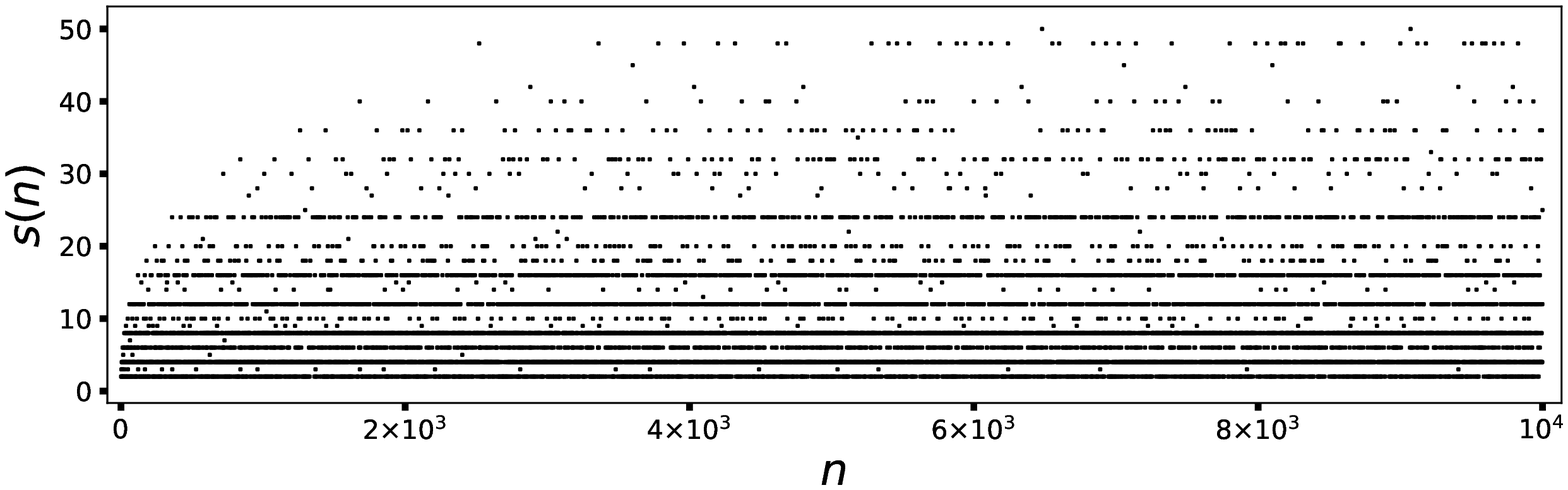} 
\caption{  }\label{fig:2a}
\end{subfigure}
\begin{subfigure}{0.49\textwidth}
\includegraphics[width=1\textwidth]{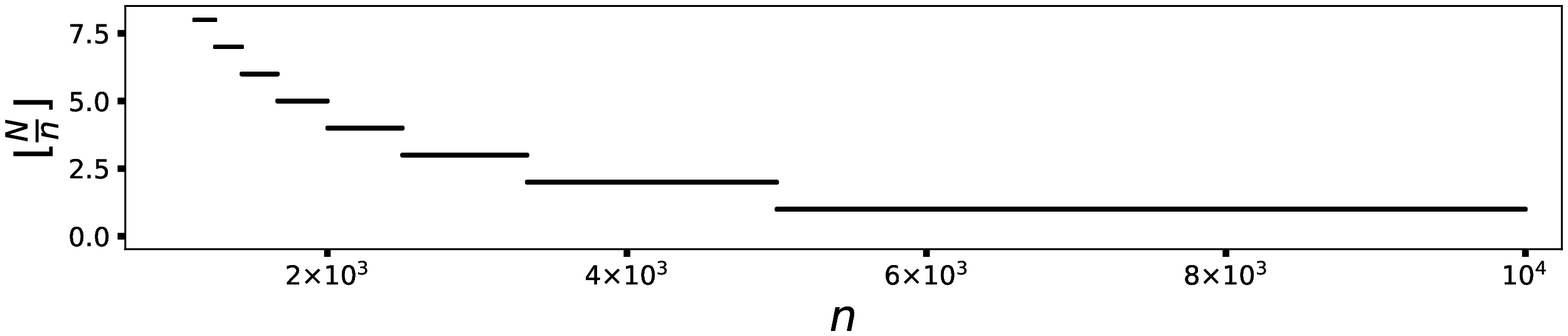}
\caption{  }\label{fig:2b}
\end{subfigure}
\vspace*{\fill} 
\begin{subfigure}{0.5\textwidth}
\includegraphics[width=1\textwidth]{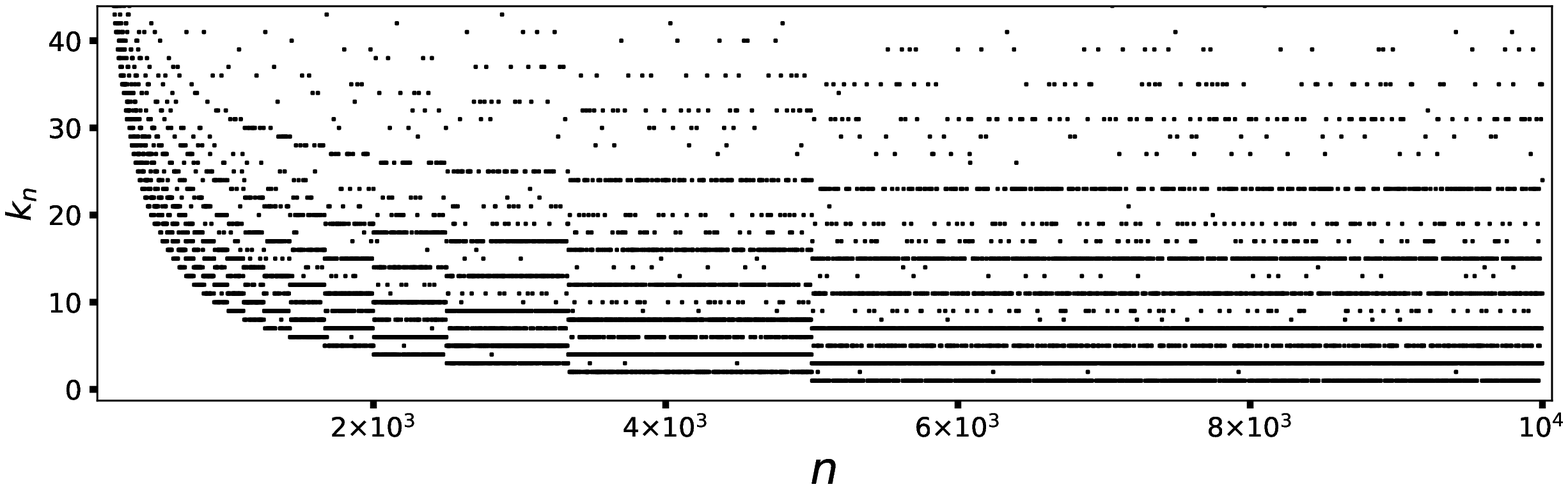}
\caption{  }\label{fig:2c}
\end{subfigure}
\caption{Plots of ($a$) divisor function, $s(n)$ vs $n$, ($b$) $\fl*{\frac{N}{n}}$ vs $n$, ($c$) degree, $k_n = s(n) + \fl*{\frac{N}{n}} - 2$ vs $n$ for $N=10^4$.
The patterns seen in the degrees $k_n$ of n 
can be related to the patterns in divisor function and the floor function.}\label{fig:2}
\end{figure}

\begin{definition}
Two vertices are said to be neighbors if there is a edge connecting them. For example, in $G_N$ the vertex $1$ is always a neighbor to any other vertex. The number of neighbors of $n$ is given by its degree $k_n$. Let $e_n$ be the number of pairs $(i,j)$ of neighbors of vertex $n$ such that $i$ and $j$ are connected. For example, consider the vertex $n=6$ in the graph $G_{10}$. The neighbors of $6$ are $1$, $2$, $3$ out of which the pairs ($1$, $2$), ($1$, $3$) are such pairs while the pair ($2$, $3$) is not. So, $e_6$ = $2$ in this graph. In terms of the Adjacency matrix

\begin{equation*}
    e_n = \sum_{s,t \in X_N} a_{ns} a_{st} a_{tn}.
\end{equation*}

That is $e_n$ is the number of neighbors $s,t$ of $n$ such that $s,t$ are connected. The local clustering coefficient $c_n$ of a vertex $n$, is defined as the ratio of $e_n$ to the number of pairs of neighbors of $n$.   
\begin{equation*}
    c_n = \frac{e_n}{^{k_n}C_2}
\end{equation*}
In the previous example, the local clustering coefficient of $6$ is $c_6$ = $\frac{2}{3}$. $c_n$ gives the fraction of the pairs of neighbors of $n$ that have an edge between them.
\end{definition}
\begin{theorem}
Let $n \in X_N$. If $n = \prod_{\alpha = 1}^k p^{j_\alpha}_\alpha$ is the prime factorisation of n and $s(n) = \prod_{\alpha = 1}^k (j_\alpha + 1)$ the number of divisors of $n$ including $n$ itself. Then 
\begin{equation}
    e_n = \sum_{m|n} s(m) -2s(n)+1 \ + \ \sum_{j=2}^{M(n)} \fl*{\frac{M(n)}{j}} + \ (M(n)-1) \ ( \ s(n)-2 \ ) 
\end{equation}
where $M(n)$ = $\fl*{\frac{N}{n}}$ .
\end{theorem}
\begin{proof}
Let $N_{s(n)-1}$ denote the number of edges among the $s(n)-1$ divisors of $n$. If the only divisor of $n$ is $1$, then $N_{s(n)-1}=0$.  Let $N_{M(n)}$ denote the number of edges among the $M(n)-1$ multiples of $n$ and finally, $N_{int}$ denote the number of edges connecting a multiple of $n$ to a divisor of $n$. Clearly, $e_n$ =  $N_{s(n)-1}$ + $N_{M(n)}$ + $N_{int}$ . 

Let $m$, $m'$ be two divisors of $n$. Then, we can write $m = \prod_{\alpha = 1}^k p^{l_\alpha}_\alpha$ and $m' = \prod_{\alpha = 1}^k p^{l_\alpha'}_\alpha$ where $0 \le l_\alpha, l_\alpha' \le j_\alpha $. If $m'$ is a divisor of $m$ then $l_\alpha \le l_\alpha'$ for each $\alpha = 1, \dots, k$. For fixed $m$, the number of divisors of m is $s(m)-1$ = $\prod_{\alpha=1}^k (l_\alpha + 1) - 1$ . Summing over all the divisors m of n, we get $\displaystyle{\sum_{m|n} (s(m) - 1)}$. But this counts
 the case when $m=m'$ twice for each divisor $m$. So, subtracting the number of divisors from this gives
\begin{equation*}
N_{s(n)-1} = \sum_{m|n} (s(m) - 1) - (s(n)-1)
\end{equation*}
\begin{equation*}
\  = \left( \sum_{l_1 = 0}^{j_1}\sum_{l_2 = 0}^{j_2} \dots \sum_{l_k = 0}^{j_k} \left[\prod_{\alpha=1}^k (l_\alpha + 1) - 1\right] \right) - (s(n)-1).
\end{equation*}
\begin{equation*}
\ = \frac{\displaystyle{\prod_{i=1}^k} \  (j_i+1)(j_i+2)}{2^k} - 2s(n)+1
\end{equation*}

The multiples of $n$ are $a_2 =(2.n), \  a_3=(3.n), \  \dots, \  a_{M(n)} = (M(n).n)$. There is an edge between $a_\alpha$ and $a_\beta$ if either $\alpha|\beta$ or $\beta|\alpha$. To find the number of such edges (say $N_{M(n)}$ ),let us look at a divisibility graph $G_{M(n)}$ with vertices $V(G_{M(n)}) = \{1,\dots,M(n)\}$. Observe that the number of such edges $N_{M(n)}$ is the number of edges in the graph $G_{M(n)}$ after removing all the edges connected to $1$. 

Since each vertex $j \in V(G_{M(n)})$ is connected to $\fl*{\frac{M(n)}{j}}-1$ number of multiples, $\displaystyle{\sum_{j=1}^{M(n)} \left[\fl*{\frac{M(n)}{j}}-1\right]}$ gives the total number of edges in $G_{M(n)}$. Removing all the edges connected to the vertex $1$, 

\begin{equation*}
N_{M(n)} = \sum_{j=2}^{M(n)} \left[\fl*{\frac{M(n)}{j}}-1\right] = \sum_{j=2}^{M(n)} \fl*{\frac{M(n)}{j}} - (M(n) - 1).
\end{equation*}

Each divisor of $n$ is also a divisor of any multiple of $n$. So, there is an edge connecting each divisor of $n$ to each of its multiple. Hence, 
\begin{equation*}
N_{int} = (M(n)-1)(s(n)-1)
\end{equation*}

Therefore,
\begin{equation*}
e_n = \sum_{m|n} s(m) -2s(n)+1 + \sum_{j=2}^{M(n)} \fl*{\frac{M(n)}{j}} + (M(n)-1)(s(n)-2)
\end{equation*}
where 
\begin{equation*}
\sum_{m|n} \  s(m) = \frac{\displaystyle{\prod_{i=1}^k} \  (j_i+1)(j_i+2)}{2^k} 
\end{equation*}
\end{proof}{}

\begin{corollary}
For a prime number $p$, $e_p$ = $\displaystyle{\sum^{M(p)}_{j=2} \fl*{\frac{M(p)}{j}}}$.
\end{corollary}

Some interesting consequences of the above corollary are : For a given $N$, if p is a prime such that $\fl*{\frac{N}{2}} < p \le N$, then $c_p = 0$. This is because 1 is its only neighbor, making $e_p=0$. Similarly, if p is a prime such that $\fl*{\frac{N}{3}} < p \le \fl*{\frac{N}{2}}$, then $c_p = 1$. This is because such a prime's neighbors are 1 and 2p which are connected making $c_p=1$.

\begin{definition}
The length of the shortest path(s) between two vertices $n,m$ in a graph is defined as the geodesic distance or the shortest distance $d_{nm}$ between the two vertices. For example, $d_{1m}$ = $1$ for any vertex $m$ in the $G_N$. If there is no path between two vertices $n,m$ then $d_{nm}=0$.  The mean geodesic distance $l_n$ of a vertex $n$ is defined as the average of the shortest distance between $n$ and all vertices $m$ in the graph. 
\begin{equation*}
l_n = \sum_{m \in X_N} \frac{d_{nm}}{N}
\end{equation*}
\end{definition}

\begin{theorem}
Let $n \in X_N$ and $k_n$ be the degree of $n$. Then, the mean geodesic distance of n, is
\begin{equation}
    l_n = \frac{2N-k_n-2}{N}
\end{equation}
\end{theorem}
\begin{proof}
Observe that 

$d_{nm}$ = 
$\begin{cases}
1, & \text{if $n \ne m$ and $n$ divides $m$ or $m$ divides $n$} 
\\0, & \text{if $n$=$m$} 
\\2, & \text{otherwise} 
\end{cases}$

This is because there is always a path of length two between any two vertices via $1$.
Therefore, $ l_n =\frac{k_n+2(N-k_n-1)}{N}= \frac{2N-k_n-2}{N}$ .
\end{proof}

\begin{corollary}
For a prime number p, 
\begin{equation}
l_p = \frac{ 2N - \fl*{\frac{N}{p}}-2 }{N} 
= 2 - \frac{2}{N} - \frac{1}{N}\fl*{\frac{N}{p}} 
\end{equation}
\end{corollary}

\begin{definition}
Let $n,s,t \in V(G)$ where $G$ is any graph.
Define $g_{st}$ as the number of geodesic paths between the vertices $s$ and $t$. Let $n_{st}^n$ be the number of geodesic paths between the vertices $s$ and $t$ which pass through the vertex $n \ne s,t$. 
The betweenness centrality of a vertex $n$ is denoted by $x_n$. It is defined as 

\begin{equation*}
x_n' = \frac{1}{(N-1)(N-2)} \sum_{s,t \ne n \in V(G)} \frac{n_{st}^n}{g_{st}}
\end{equation*}

Let $x_n = (N-1)(N-2)x_n'$

In general, calculation of $x_n$ for a given graph is computationally done using different algorithms. However, in the case of the $G_N$, the values of $g_{st}$, $n_{st}^n$ can be calculated directly from the Adjacency matrix. 

Assume, $s,t,n \in X_N$. $g_{st}$ is the number of geodesic paths between $s$ and $t$ and $n_{st}^n$ is the number of geodesic paths between $s$ and $t$ which pass through $n \ne s,t$. If $s$ and $t$ are connected, then the geodesic path is of length one and it cannot pass through another $n$. So, $n_{st}^n$ = 0, if $s$ and $t$ are connected. If $s$ and $t$ are not connected then, geodesic path is of length 2, since there is a path between any two vertices in $G_N$ via the vertex $1$. The number of paths of length 2 between $s$ and $t$ is given by $[a^2]_{st} = \displaystyle{\sum_{\alpha=1}^N} a_{s\alpha}a_{\alpha t}$ . Hence, $g_{st}$ = $[a^2]_{st}$, if $s$ and $t$ are not connected. And $n_{st}^n$ can be written in terms of the adjacency matrix as $n_{st}^n$ = $a_{sn}a_{nt}$, which is one only if there is a path of length two via $n$ between $s$ and $t$. 

We can write $x_n$ as

\begin{equation}
    x_n = \sum_{s,t \ne n \in X_N} \frac{(1-a_{st})a_{sn}a_{nt}}{[a^2]_{st}}.
\end{equation}

The numerator is non-zero only if $a_{st}$ is zero. This gives a way to compute $x_n$ for $G_N$ directly from the Adjacency matrix.
\end{definition}

\begin{theorem}
Let $s,t,n \in X_N$ such that there is no edge between $s,t$. If $n = \prod_{\alpha = 1}^k p^{j_\alpha}_\alpha$ is the prime factorisation of n and $s(n) = \prod_{\alpha = 1}^k (j_\alpha + 1)$ the number of divisors of n including itself. Denote gcd of two numbers $s,t$ by d(s,t) and their lcm by l(s,t). Then,

\begin{equation}
   g_{st} =  s( d(s,t) ) + \fl*{ \frac{N} {st/d(s,t)} }
\end{equation}

\end{theorem}

\begin{proof}
$g_{st}$ gives the number of paths of length two between $s$ and $t$ when they are not connected. There is a path of length two between $s,t$ via each of the common multiples of $s,t$ and also via the common divisors of $s,t$ including the vertex $1$. Number of common multiples of $s,t$, is the number of multiples of $l(s,t)$ which is $\fl*{ \frac{N} {l(s,t)} }$. Since $st = l(s,t) d(s,t)$, we can write $\fl*{\frac{N} {l(s,t)}}   = \fl*{ \frac{N} {st/d(s,t)} }$ to get an expression which depends only on gcd.  Number of common divisors of $s,t$ is the number of divisors of $d(s,t)$ given by $s(d(s,t))$. Hence, $g_{st} =  s( d(s,t) ) + \fl*{ \frac{N} {st/d(s,t)} }$.

\end{proof}

\begin{corollary}
\begin{equation}
    x_n =  \sum_{s,t \in X_N} \frac{( 1 - a_{st} ) a_{sn} a_{nt} }  { s( d(s,t) ) + \fl*{ \frac{N} {st/d(s,t)} } } 
\end{equation}
\end{corollary}

\begin{corollary}
For a prime number p, 
\begin{equation}
    x_p = \sum_{j = 2}^{M(p)} \sum_{k = j + 1, j \nmid k}^{M(p)} \frac{1}{ s( d(j,k) p ) + \fl*{ \frac{N} { pjk / d(j,k) } } } 
\end{equation}
where $M(p)=\fl*{\frac{N}{p}}$
\end{corollary}

\begin{corollary}

\begin{equation}
  \sum_{s,t \in X_N} n_{st}^n = {}^{k_n} C_2 - e_n     
\end{equation}
gives the number of pairs neighbors of $n$ between which there is no edge. 
\end{corollary}

\subsection{Connectance of $G_N$}
\begin{definition}
Connectance, $C$ (or link density) of a graph is defined as the ratio of the number of edges present in the graph to the number of possible edges between all vertex pairs n,m in the graph of size $N$.
\end{definition}

\begin{theorem}
The connectance, $C$ of $G_N$ is $\displaystyle{\frac { NlnN+2(\gamma-1)N+O(N^{0.5}) } {^NC_2}}$ .
\end{theorem}
\begin{proof}
By definition, connectance $C = \displaystyle{\sum_{n,m \in V(G_N)} \frac{0.5A_{nm}}{^NC_2} = \sum_{n=1}^N \frac{ 0.5k_n }{  ^NC_2}}$.
Now, $\displaystyle{\sum_{n=1}^N} k_n = \displaystyle{\sum_{n=1}^N} (\fl*{\frac{N}{n}} + s(n) - 2) $.
Since, $\displaystyle{\sum_{n=1}^N} \fl*{\frac{N}{n}} = \displaystyle{\sum_{n=1}^N} s(n)$,
\begin{equation*}
\displaystyle{\sum_{n=1}^N} k_n = \displaystyle{\sum_{n=1}^N} (2 \fl*{\frac{N}{n}} - 2) \end{equation*}
Therefore,
\begin{equation*}
C = \sum_{n=1}^N \frac{ \fl*{\frac{N}{n}} -1 }{  ^NC_2} = \sum_{n=1}^N \frac { \fl*{\frac{N}{n}} } {^NC_2}  - \frac{N}{^NC_2}
\end{equation*}
\begin{equation*}
=\frac { NlnN+(2\gamma-1)N+O(N^{0.5}) } {^NC_2}  - \frac{N}{^NC_2}
\end{equation*}
            
\begin{equation}
\text{Connectance}    = \displaystyle{\frac { NlnN+2(\gamma-1)N+O(N^{0.5}) } {^NC_2}} .
\end{equation}
\end{proof}

\section{ Patterns in $\Delta c_n$ for varying $n$}
\begin{definition}
Let $\Delta c_n$ = $c_n - c_{n+1}$ be the difference in the local clustering coefficients of consecutive vertices.
Consider $n,n+1$ such that $\fl*{\frac{N}{a+1}} < n < n+1 \le \fl*{\frac{N}{a}}$ for some positive integer $a$, then $M(n) = M(n+1) = a$.
Suppose $\displaystyle{\sum_{m|n} s(m)  = \sum_{m'|n+1} s(m') }$, then $ s(n) = s(n+1) $, which results in $\Delta c_n = 0$.

In terms of the expression,
$$
\Delta c_n = \frac{\displaystyle{\sum_{m|n} s(m) } - 2s(n) + \sum_{j=2}^{M(n)} \fl*{\frac{M(n)}{j}} + (M(n)-1)(s(n)-2)}{0.5k_n(k_n-1)} 
$$
$$
- \   \frac{\displaystyle{\sum_{m'|n+1} s(m')} - 2s(n+1) + \sum_{j=2}^{M(n+1)} \fl*{\frac{M(n+1)}{j}} + (M(n+1)-1)(s(n+1)-2)}{0.5k_{n+1}(k_{n+1}-1)}.
$$

Let us take $n,n+1$ such that $\fl*{\frac{N}{a+1}} < n < n+1 \le \fl*{\frac{N}{a}}$ for some positive integer $a$, then $M(n) = M(n+1) = a$.

$$\Delta c_n = \frac{\left(\displaystyle{\prod_{i=1}^k} \  (j_i+1)(j_i+2)/2^k \right) - 2s(n) + (a-1)(s(n)-2)}{0.5(a + s(n) - 2)(a + s(n) - 3)}
$$
$$
- \  \frac{\left(\displaystyle{\prod_{i=1}^{k'} \  (j_i'+1)(j_i'+2)/2^{k'}}\right) - 2s(n+1) + (a-1)(s(n+1)-2)}{0.5(a + s(n+1) - 2)(a + s(n+1) - 3)}.$$

We see that $\Delta c_n = 0$ when 
\begin{eqnarray}
\frac{\displaystyle{\prod_{i=1}^k} \  (j_i+1)(j_i+2)}{2^k}  = \frac{\displaystyle{\prod_{i=1}^{k'}} \  (j_i'+1)(j_i'+2)}{2^{k'}} \label{5.130}
\end{eqnarray} 
and
\begin{eqnarray} 
s(n) = s(n+1) \label{5.140} 
\end{eqnarray} 
i.e, $\Delta c_n = 0$ when $n,n+1$ have the same prime powers and the same number of divisors. A result by Heath-Brown (1984) showed that there are infinitely many consecutive natural numbers $n,n+1$ such that $s(n)=s(n+1)$ \cite{He84}. It is possible that there are also infinitely many consecutive pairs $n,n+1$ with not only the same number of divisors, but exactly the same prime powers. For example, in a graph of size of $N=100$, $93$ \& $94$ is a pair that lies between $\fl*{\frac{N}{2}}$ and $\fl*{\frac{N}{1}}$ and have the same $c_n = \frac{2}{3}$. Therefore $\Delta c_{93} = 0$. Other such pairs are $94$ \& $95$, $86$ \& $87$ and $85$ \& $86$ which have the same $c_n = \frac{2}{3}$, resulting in $\Delta c_n = 0$.  Similarly, there are many such pairs depending on the size of the graph.\\
\end{definition}

\begin{figure}[hbtp]
\includegraphics[width=0.5\textwidth]{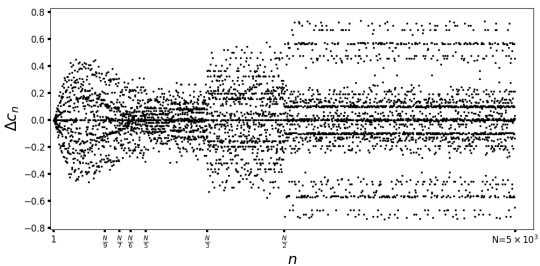}
\caption{Plot of $\Delta c_n$ vs $n$ for $N = 5 \times 10^3$. }\label{fig:3}
\end{figure}

One of the authors studied numerically the plot of $\Delta c_n$ vs $n$ for large N, in reference \cite{Snehal15}. In figure \ref{fig:3} we show the patterns in the plot for $ N= 5 \times 10^3$. We observe that there are many pairs $n,n+1$ with $\Delta c_n \ne 0$. Some examples when $N=100$ are, $c_{82} = \frac{2}{3}$, $c_{83} = 0$, so $\Delta c_{82} = \frac{2}{3}$ and $c_{73} = 0$  $c_{74} = \frac{2}{3}$, so $\Delta c_{73} = -\frac{2}{3}$.

From this figure, we conjecture that there are infinitely many consecutive pairs $n, n+1$ such that $ s(n) = s(n+1) $ and $\displaystyle{\sum_{m|n} s(m)  = \sum_{m'|n+1} s(m') + k }$ for some $k \in  \mathbb{Z}$. We hope that further explorations in this direction, may lead to insights about how large can $k$ be and bring out deeper
connections between number theory and graph theory along similar lines.

\section*{Acknowledgements}
One of the authors, Abiya R would like to thank Department of Science and Technology, Govt. of India for INSPIRE scholarship.
\section*{Conflict of interest}
The authors certify that they have no affiliations with or involvement in any organization or entity with any financial interest or non-financial interest in the subject matter or materials discussed in this manuscript.

\end{document}